\documentclass[review]{siamart190516}   

\usepackage{siampaper}
\newcommand{\USYMQR}{\textsc{Usymqr}\xspace}
\newcommand{\USYMLQ}{\textsc{Usymlq}\xspace}

\newcommand{\USYMLQR}{\textsc{Usymlqr}\xspace}

\newcommand{\BiLQ}{\textsc{BiLQ}\xspace}
\newcommand{\BiCG}{\textsc{BiCG}\xspace}

\newcommand{\QMR}{\textsc{Qmr}\xspace}

\newcommand{\CGM}{\textsc{Cg}\xspace}

\newcommand{\FOM}{\textsc{Fom}\xspace}
\newcommand{\GMRES}{\textsc{Gmres}\xspace}

\newcommand{\DQGMRES}{\textsc{Dqgmres}\xspace}

\newcommand{\BLCG}{\textsc{Block-Cg}\xspace}
\newcommand{\BLMINRES}{\textsc{Block-Minres}\xspace}

\newcommand{\BLGMRES}{\textsc{Block-Gmres}\xspace}

\newcommand{\SPMR}{\textsc{Spmr}\xspace}

\newcommand{\TriCG}{\textsc{TriCG}\xspace}
\newcommand{\TriMR}{\textsc{TriMR}\xspace}

\newcommand{\GPCG}{\textsc{Gpcg}\xspace}
\newcommand{\GPMR}{\textsc{Gpmr}\xspace}

\newcommand{\GMRESk}{\textsc{Gmres}($k$)\xspace}

\def\T{^T\!}

\newlength{\forwidth}
\newcommand{\bmat}[1]{\begin{bmatrix} #1 \end{bmatrix}}

\usepackage[utf8]{inputenc}
\usepackage[T1]{fontenc}

\graphicspath{{graphics/}}  
\usepackage{tikzscale}
\usepackage{pgfplots}
\pgfplotsset{compat=newest}
\usetikzlibrary{external}
\usetikzlibrary{backgrounds}
\usetikzlibrary{intersections}
\usetikzlibrary{math}

\newcommand*{\includetikzgraphics}[2][]{%
  \includegraphics[#1]{#2}
}

\makeatletter
\renewcommand{\todo}[2][]{\tikzexternaldisable\@todo[#1]{#2}\tikzexternalenable}
\makeatother

\renewcommand{\year}{2021}      
\newcommand{\cahiernumber}{62}  

\algsetblock[Name]{For}{EndFor}{}{.75em}
\algsetblock[Name]{While}{EndWhile}{}{.75em}
\usepackage{kbordermatrix}
\usepackage{multirow}

\newcommand{\barbar}[1]{\bar{\bar{#1}}}
\newcommand{\x}{\boldsymbol{\times}}

\newcommand{\applyGivens}{\mathop{\text{ref}}}
\newcommand{\computeGivens}{\mathop{\text{givens}}}

\newlength{\outlength}
\newcommand{\ain}{a^{\text{\makebox[\outlength][l]{in}}}}
\newcommand{\ainshort}{a^{\text{in}}}
\newcommand{\aout}{a^{\text{\makebox[\outlength][l]{out}}}}

\title{%
  {\GPMR}:\@ An Iterative Method for Unsymmetric Partitioned Linear Systems
}

\author{%
  Alexis Montoison%
  \thanks{%
    GERAD and Department of Mathematics and Industrial Engineering,
    Polytechnique Montr\'eal, QC, Canada.
    E-mail: \mailto{alexis.montoison@polymtl.ca}.
    Research supported by a FRQNT grant and an excellence scholarship of the IVADO institute.
  }
  \and
  Dominique Orban%
  \thanks{%
    GERAD and Department of Mathematics and Industrial Engineering,
    Polytechnique Montr\'eal, QC, Canada.
    E-mail: \mailto{dominique.orban@gerad.ca}.
    Research partially supported by an NSERC Discovery Grant.
  }
}
\date{\today}

\begin{document}

  \nolinenumbers
  \maketitle

  \thispagestyle{firstpage}
  \pagestyle{myheadings}

  \begin{abstract}
    We introduce an iterative method named \GPMR for solving \(2\)\(\times\)\(2\) block unsymmetric linear systems.
    \GPMR is based on a new process that reduces simultaneously two rectangular matrices to upper Hessenberg form and that is closely related to the block-Arnoldi process.
    \GPMR is tantamount to \BLGMRES with two right-hand sides in which the two approximate solutions are summed at each iteration, but requires less storage and work per iteration.
    We compare the performance of \GPMR with \GMRES and \BLGMRES on linear systems from the SuiteSparse Matrix Collection.
    In our experiments, \GPMR terminates significantly earlier than \GMRES on a residual-based stopping condition with an improvement ranging from around 10\% up to 50\% in terms of number of iterations.
    We also illustrate by experiment that \GPMR appears more resilient to loss of orthogonality than \BLGMRES.
  \end{abstract}



  \section{Introduction}

  Consider the partitioned linear system
  \begin{equation}
    \label{eq:gsp}
    \bmat{
      M         & A_{\star} \\
      B_{\star} & N
    }
    \bmat{
      x_{\star} \\ y_{\star}
    }
    =
    \bmat{
      b_{\star} \\ c_{\star}
    },
  \end{equation}
  where $M \in \R^{m \times m}$, $N \in \R^{n \times n}$, $A_{\star} \in \R^{m \times n}$ and $B_{\star} \in \R^{n \times m}$.
  We assume that $A_{\star}$ and $B_{\star}$ are nonzero, and that $b_{\star} \in \R^m$ and $c_{\star} \in \R^n$ are both nonzero.
  System~\eqref{eq:gsp} occurs, among others, in the discretization of systems of partial-differential equations, including the Navier-Stokes equations by way of the finite elements method \citep{elman-2002}.
  A prime example is domain decomposition with no overlap, also known as iterative substructuring \citep{dolean-jolivet-nataf-2015}, that consists in splitting a domain into $k$ non-overlapping subregions, and that leads to structured matrices with \emph{arrowhead} form \citep{ferris-horn-1998}.
  Let $\mathcal{I}$ be the set of all indices of the discretization points that belong to the interior of the subdomains and $\Gamma$ the set of those corresponding to the interfaces between the subdomains.
  Grouping the unknowns corresponding to $\mathcal{I}$ by subdomain in $u_{\mathcal{I}}$ and those corresponding to $\Gamma$ in $u_{\Gamma}$, we obtain the arrowhead partitioning of the stiffness system
  \begin{equation}
    \label{eq:ism}
    \bmat{A_{\mathcal{I}\mathcal{I}} & A_{\mathcal{I} \Gamma} \\ A_{\Gamma \mathcal{I}} & A_{\Gamma\Gamma}}
    \bmat{u_{\mathcal{I}} \\ u_\Gamma}
    =
    \bmat{f_{\mathcal{I}} \\ f_\Gamma}
    \quad \Longleftrightarrow \quad
    \bmat{A_{11}       &        &              & A_{1 \Gamma} \\
                       & \ddots &              & \vdots       \\
                       &        & A_{kk}       & A_{k \Gamma} \\
          A_{\Gamma 1} & \hdots & A_{\Gamma k} & A_{\Gamma \Gamma}}
    \bmat{u_1 \\ \vdots \\ u_{k} \\ u_\Gamma}
    =
    \bmat{f_1 \\ \vdots \\ f_{k} \\ f_\Gamma},
  \end{equation}
  where \(u = (u_{\mathcal{I}}, u_\Gamma)\) is the vector of nodal displacements and \(f\) the vector of nodal forces.
  For a tour of applications leading to~\eqref{eq:gsp}, we refer the reader to \citep{benzi-golub-liesen-2005}.
  We assume that there exist nonsingular $P_{\ell}$ and $P_r$ with inexpensive inverses such that
  \begin{equation}
    \label{eq:K}
    K :=
    P_{\ell}^{-1}
    \bmat{
      M         & A_{\star} \\
      B_{\star} & N
    }
    P_r^{-1}
    =
    \bmat{
      \lambda I & A     \\
      B         & \mu I
    }
    ,\quad \lambda,\mu \in \R,
  \end{equation}
  so that the equivalent preconditioned system
  \begin{equation}
    \label{eq:gsp_preconditioned}
    \bmat{
      \lambda I & A     \\
      B         & \mu I
    }
    \bmat{
    x \\ y
    }
    =
    \bmat{
    b \\ c
    },
    \quad
    \bmat{x_{\star} \\ y_{\star}} = P_{r}^{-1} \bmat{x \\ y},
    \quad
    \bmat{b \\ c} = P_{\ell}^{-1} \bmat{b_{\star} \\ c_{\star}}
  \end{equation}
  can be solved instead of~\eqref{eq:gsp}.
  Note that \(\lambda\) and/or \(\mu\) may vanish.
  For example, the ideal preconditioners of \citet{murphy-golub-wathen-2000} and \citet{ipsen-2001} lead to~\eqref{eq:K}.
  Although ideal preconditioners are typically impractical because they require the solution of systems with the Schur complement $S = N - B_{\star}M^{-1\!}A_{\star}$, viable preconditioners such that $P_{\ell} P_r = \blkdiag(M, N)$ can be employed when $M$ and $N$ are both nonsingular.

  Given an unstructured matrix $C$, a practical approach to recovering the matrix of~\eqref{eq:gsp} is to
  permute its rows and columns with orderings determined by graph partitioning tools such as METIS \citep{karypis-kumar-1997}.
  This reordering also provides a uniform partitioning to compute a parallel block-Jacobi preconditioner for~\eqref{eq:K}.

  When \(\lambda \neq 0\),~\eqref{eq:gsp_preconditioned} can be reduced to the Schur complement system
  \[
    (\mu I - \lambda^{-1} BA) y = c - \lambda^{-1} Bb, \quad x = \lambda^{-1} (b - Ay).
  \]
  Such eliminated system is attractive because of its smaller size, but may have worse conditioning than~\eqref{eq:gsp_preconditioned}, e.g., when \(B = A^T\), \(M = M^T \succ 0\) and \(N = N^T \preceq 0\), though not always, e.g, when~\eqref{eq:gsp} is symmetric and positive definite.
  In this paper, we focus on applying an iterative method to~\eqref{eq:gsp_preconditioned} directly while exploiting its block structure.

  \subsection*{Contributions}

  Our main contributions are (i) a new orthogonal Hessenberg reduction process, (ii) an iterative method based on said process named \GPMR (General Partitioned Minimal Residual) specialized for~\eqref{eq:gsp_preconditioned}, and (iii) an efficient software implementation to solve~\eqref{eq:gsp_preconditioned} in arbitrary floating-point arithmetic on CPU and GPU.

  \subsection*{Related research}

  Numerous Krylov methods have been developed for solving general unsymmetric linear systems, including \BiLQ \citep{montoison-orban-2020}, \GMRES \citep{saad-schultz-1986}, or \QMR \citep{freund-nachtigal-1991}.
  Few are tailored specifically to the block structure of~\eqref{eq:gsp}.

  Specialized iterative methods have been developed for special cases of~\eqref{eq:gsp}.
  \citet{estrin-greif-2018} developed \SPMR; a family of methods for~\eqref{eq:gsp} that exploit its block structure when $N = 0$ and $b$ or $c$ is zero.
   \citet{buttari-orban-ruiz-titley_peloquin-2019} developed \USYMLQR, an interlacing of the methods \USYMLQ and \USYMQR of \citet{saunders-simon-yip-1988}, applicable when \(A = B^T\), \(M = M^T \succ 0\) and \(N = 0\).
  \citet{greif-wathen-2019} formulate conditions under which \CGM may be used in the case where \(M \succeq 0\) is maximally rank deficient and \(N = N^T \preceq 0\).
   When \(N = N^T \prec 0\) also holds, \citet{orban-arioli-2017} propose a family of methods inspired from regularized least norm and least squares that apply after a translation so that either \(b\) or \(c\) is zero, and \citet{montoison-orban-2021} develop \TriCG and \TriMR, two methods related to \BLCG and \BLMINRES.
   When \(A = B^T\), and \(M\) and \(N\) are either zero or symmetric definite matrices, our orthogonal Hessenberg reduction process coincides with that of \citet{saunders-simon-yip-1988} and \GPMR coincides with \TriMR in exact arithmetic.

  \subsection*{Notation}

  All vectors are columns vectors.
  Vectors and matrices are denoted by lowercase Latin and capital Latin letters, respectively.
  The only exceptions are \(2\)\(\times\)\(2\) blocks, which are represented by capital Greek letters, and the matrices denoted $w_k$ below.
  For a vector $v$, $\|v\|$ denotes the Euclidean norm of $v$, and for a matrix $M$, $\|M\|_F$ denotes the Frobenius norm of $M$.
  The shorthand $y \mapsto M \backslash y$ represents an operator that returns the solution of $Mx = y$.
  $e_i$ is the $i$-th column of an identity matrix of size dictated by the context.
  \(I_k\) represents the \(k\)\(\times\)\(k\) identity operator.
  We omit the subscript \(k\) when it is clear from the context.
  We let
  \begin{equation}
    \label{eq:def-K-K0-D}
    K_0 :=
    \begin{bmatrix}
      0 & A \\
      B & 0
    \end{bmatrix},
    \quad
    \blkdiag(\lambda I, \mu I) =
    \begin{bmatrix}
      \lambda I & 0     \\
      0         & \mu I
    \end{bmatrix},
    \quad
    d :=
    \bmat{b \\ c},
    \quad
    D :=
    \bmat{b & 0 \\ 0 & c}.
  \end{equation}
  For a matrix $C$ and a vector $t$, $\mathcal{K}_k(C, t)$ is the Krylov subspace $\Span\left\{t, Ct, \ldots, C^{k-1}t\right\}$.
  For a matrix $T$ with as many rows as $C$ has columns, $\mathcal{K}_k(C, T)$ is the block-Krylov subspace $\Span\left\{T, CT, \ldots, C^{k-1}T\right\}$.
  We abusively write $(b, c)$ and $l = (l_1, \ldots, l_n)$ to represent the column vectors $\bmat{b\T & c\T\,}^T$ and $l = \bmat{l_1 & \cdots & l_n}\T$, respectively.

  \section{A Hessenberg reduction process}

  In this section, we state a new Hessenberg reduction process for general $A$ and $B$, its relationship with the block-Arnoldi process, and the modifications necessary for regularization.

  \begin{theorem}
    \label{theorem:mo}
    Let $A \in \R^{m \times n}$, $B \in \R^{n \times m}$, and $p := \min \{m,n\}$.
    There exist $V \in \R^{m \times p}$ and $U \in \R^{n \times p}$ with othonormal columns, and upper Hessenberg $H \in \R^{p \times p}$ and $F \in \R^{p \times p}$ with nonnegative subdiagonal coefficients such that
    \begin{subequations}
      \begin{align}
        \label{eq:VAU}
        V\T A U & = H,\\
        \label{eq:UBV}
        U\T B V & = F.
      \end{align}
    \end{subequations}
  \end{theorem}

  \begin{proof}
    Choose arbitrary unit $u_1 \in \R^n$ and $v_1 \in \R^m$.
    For $k = 1, \ldots, p-1$, define
    \begin{subequations}
      \label{eq:def_uv}
      \begin{align}
        \beta_{k+1}  v_{k+1} & = A u_k - \textstyle\sum_{i=1}^k (v_i\T A u_k) v_i, \\
        \gamma_{k+1} u_{k+1} & = B v_k - \textstyle\sum_{i=1}^k (u_i\T B v_k) u_i,
      \end{align}
    \end{subequations}
    with positive $\beta_{k+1}$ and $\gamma_{k+1}$ such that $v_{k+1}$ and $u_{k+1}$ are unit vectors.
    In case of breakdown, which happens if $Au_k \in \Span\{v_1,~\ldots,~v_k\}$ or $Bv_k \in \Span\{u_1,~\ldots,~u_k\}$, we choose an arbitrary unit $v_{k+1} \perp \Span\{v_1,~\ldots,~v_k\}$ or $u_{k+1} \perp \Span\{u_1,~\ldots,~u_k\}$ and set $\beta_{k+1} = 0$ or $\gamma_{k+1} = 0$, respectively.
    We prove by induction that the following statement, denoted $\mathcal{P}(k)$, is verified:
    \begin{equation}
      \label{eq:hessenberg_reduction}
        v_j\T v_{k+1} = 0
        \quad \text{and} \quad
        u_j\T u_{k+1} = 0 \quad
        (j = 1, \ldots, k).
    \end{equation}
    In view of the above, \(v_j\T v_{k+1} = 0\) clearly holds if \(\beta_{k+1} = 0\), while \(u_j\T u_{k+1} = 0\) holds if \(\gamma_{k+1} = 0\).
    Thus we focus on the case where~\eqref{eq:def_uv} applies.
    Because $v_1$ and $u_1$ are unit vectors,
    \begin{alignat*}{3}
      \beta_2  v_1\T v_2 & = v_1\T A u_1 - (v_1\T A u_1) v_1\T v_1 && = (1 - \|v_1\|^2)(v_1\T A u_1) && = 0, \\
      \gamma_2 u_1\T u_2 & = u_1\T B v_1 - (u_1\T B v_1) u_1\T u_1 && = (1 - \|u_1\|^2)(u_1\T B v_1) && = 0,
    \end{alignat*}
    so that the base case $\mathcal{P}(1)$ holds.
    Let $\mathcal{P}(1), \ldots, \mathcal{P}(k-1)$ hold.
    For $j = 1, \ldots, k$,~\eqref{eq:def_uv} implies
    \begin{alignat*}{3}
      \beta_{k+1}  v_j\T v_{k+1} & = v_j\T A u_{k} - \textstyle\sum_{i=1}^{k} (v_i\T A u_{k}) v_j\T v_i && = v_j\T A u_{k} - (v_j\T A u_{k}) v_j\T v_j && = 0, \\
      \gamma_{k+1} u_j\T u_{k+1} & = u_j\T B v_{k} - \textstyle\sum_{i=1}^{k} (u_i\T B v_{k}) u_j\T u_i && = u_j\T B v_{k} - (u_j\T B v_{k}) u_j\T u_j && = 0,
    \end{alignat*}
    so that $\mathcal{P}(k)$ also holds.
    For $j = 1, \ldots, k-1$, we have from~\eqref{eq:def_uv} and $\mathcal{P}(k)$ that
    \begin{alignat*}{2}
      v_{k+1}\T A u_j & = v_{k+1}\T \left(\beta_{j+1}  v_{j+1} + \textstyle\sum_{i=1}^j (v_i\T A u_j)v_i\right) && = 0, \\
      u_{k+1}\T B v_j & = u_{k+1}\T \left(\gamma_{j+1} u_{j+1} + \textstyle\sum_{i=1}^j (u_i\T B v_j)u_i\right) && = 0,
    \end{alignat*}
    because $k+1 > j+1$.
    Thus, $V := \bmat{v_1 & \ldots & v_p}$, $U := \bmat{u_1 & \ldots & u_p}$,
    \begin{equation*}
    H=
    \bmat{
      v_1\T A u_1 & v_1\T A u_2 & \hdots    & v_1\T A u_p     \\
      \beta_2     & \ddots      & \ddots    & \vdots          \\
                  & \ddots      & \ddots    & v_{p-1}\T A u_p \\
                  &             & \beta_{p} & v_p\T A u_p
    }
    \text{ and }
    F=
    \bmat{
      u_1\T B v_1 & u_1\T B v_2 & \hdots     & u_1\T B v_p     \\
      \gamma_2    & \ddots      & \ddots     & \vdots          \\
                  & \ddots      & \ddots     & u_{p-1}\T B v_p \\
                  &             & \gamma_{p} & u_p\T B v_p
    }
    \end{equation*}
    satisfy~\eqref{eq:VAU}--\eqref{eq:UBV} and have the properties announced.
  \end{proof}
  \Cref{alg:mo} formalizes a Hessenberg reduction process derived from \cref{theorem:mo}.

  \begin{algorithm}[ht]
    \caption{%
      Orthogonal Hessenberg reduction
    }
    \label{alg:mo}
    \begin{algorithmic}[1]
      \Require $A$, $B$, $b$, $c$, all nonzero
      \State\label{alg:mo:init}%
      $\beta v_1 = b$, $\gamma u_1 = c$  \Comment{$(\beta, \, \gamma) > 0$ so that $\|v_1\| = \|u_1\| = 1$}
      \For{$k$ = 1,~2,~\(\ldots\)}
        \For{$i$ = 1,~\(\ldots\),~$k$}
          \State $h_{i,k} = v_i\T A u_k$
          \State $f_{i,k} = u_i\T B v_k$
        \EndFor
        \State $h_{k+1,k} v_{k+1} = A u_k - \sum_{i=1}^k h_{i,k} v_i$ \Comment{$h_{k+1,k} > 0$ so that $\|v_{k+1}\| = 1$}
        \State $f_{k+1,k} u_{k+1} = B v_k - \sum_{i=1}^k f_{i,k} u_i$ \Comment{$f_{k+1,k} > 0$ so that $\|u_{k+1}\| = 1$}
      \EndFor
    \end{algorithmic}
  \end{algorithm}

  Define \(V_k := \begin{bmatrix} v_1 & \ldots & v_k \end{bmatrix}\) and \(U_k := \begin{bmatrix} u_1 & \ldots & u_k \end{bmatrix}\).
  After \(k\) iterations of \Cref{alg:mo}, the situation may be summarized as
  \begin{subequations}
    \label{eq:mo}
    \begin{alignat}{2}
         A U_k & = V_k H_k + h_{k+1,k} v_{k+1} e_k^T && = V_{k+1} H_{k+1,k}
         \label{eq:mo-V}
      \\ B V_k & = U_k F_k + f_{k+1,k} u_{k+1} e_k^T && = U_{k+1} F_{k+1,k}
        \label{eq:mo-U}\\
        V_k^T V_k & = U_k^T U_k = I_k,
        \label{eq:mo-UV}
    \end{alignat}
  \end{subequations}
  where
  \begin{equation*}
    H_k =
    \bmat{
      h_{1,1}~ & h_{1,2}~ & \hdots    & h_{1,k}   \\
      h_{2,1}~ & \ddots~  & \ddots    & \vdots    \\
               & \ddots~  & \ddots    & h_{k-1,k} \\
               &          & h_{k,k-1} & h_{k,k}
    }, \qquad
    F_k =
    \bmat{
      f_{1,1}~ & f_{1,2}~ & \hdots    & f_{1,k}   \\
      f_{2,1}~ & \ddots~  & \ddots    & \vdots    \\
               & \ddots~  & \ddots    & f_{k-1,k} \\
               &          & f_{k,k-1} & f_{k,k}
    },
  \end{equation*}
  and
  \begin{equation*}
    H_{k+1,k} =
    \bmat{
      H_{k} \\
      h_{k+1,k} e_{k}^T
    }, \qquad
    F_{k+1,k} =
    \bmat{
      F_{k} \\
      f_{k+1,k} e_{k}^T
    }.
  \end{equation*}
  If $B = A\T$, \Cref{alg:mo} reduces to the orthogonal tridiagonalization process of \citet{saunders-simon-yip-1988}, $H_k$ and $F_k$ are tridiagonal and $H_k = F_k^T$.
  \Cref{alg:mo} uses the Gram-Schmidt method for computing $\ell_2$-orthonormal bases $V_k$ and $U_k$ for simplicity.
  In a practical implementation, the modified Gram-Schmidt algorithm would be used instead.
  While~\eqref{eq:mo-V}--\eqref{eq:mo-U} hold to within machine precision despite loss of orthogonality,~\eqref{eq:mo-UV} holds only in exact arithmetic.
  In exact arithmetic,~\eqref{eq:mo} yields
  \[
    V_k\T A U_k = H_k \quad \text{and} \quad U_k\T B V_k = F_k,
  \]
  which imply that the singular values of \(H_k\) and \(F_k\) are estimates of those of \(A\) and \(B\), respectively.
  That is in contrast with the process of \citet{arnoldi-1951}, which can be used to approximate eigenvalues.

  \subsection{Relation with the block-Arnoldi process}%
  \label{sec:relation-arnoldi}

  For \(k \geq 1\),
  \begin{subequations}
    \label{eq:subspaces}
    \begin{align}
      v_{2k}   & \in \Span\{b, \ldots, (AB)^{k-1}b              , Ac, \ldots, (AB)^{k-1}Ac\}, \\
      v_{2k+1} & \in \Span\{b, \ldots, (AB)^{k}b \phantom{^{-1}}, Ac, \ldots, (AB)^{k-1}Ac\}, \\
      u_{2k}   & \in \Span\{c, \ldots, (BA)^{k-1}c              , Bb, \ldots, (BA)^{k-1}Bb\}, \\
      u_{2k+1} & \in \Span\{c, \ldots, (BA)^{k}c \phantom{^{-1}}, Bb, \ldots, (BA)^{k-1}Bb\}.
    \end{align}
  \end{subequations}
  The subspaces generated by \Cref{alg:mo} can be viewed as the union of two block-Krylov subspaces generated by $AB$ and $BA$ with respective starting blocks $\bmat{b & Ac}$ and $\bmat{c & Bb}$.
  Note the similarity between~\eqref{eq:mo-ba} and a Krylov process in which basis vectors have been permuted.
  Let
  $$
    P_k := \bmat{e_1 & e_{k+1} & \cdots & e_i & e_{k+i} & \cdots & e_k & e_{2k}} = \bmat{E_1 & \cdots & E_k},
    \mathhfill
    E_k := \bmat{e_k & \\ & e_k}
  $$
  denote the permutation introduced by \citet{paige-1974} that restores the order in which \Cref{alg:mo} generates basis vectors, i.e.,
  \begin{equation}
    \label{eq:form-Wk}
    W_k := \bmat{V_k & 0 \\ 0 & U_k} P_k
    =
    \bmat{
      w_1 & \cdots & w_k
    },
    \qquad
    w_k =
    \bmat{
      v_k & 0 \\ 0 & u_k
    } :=
    \bmat{ v_k^\circ & u_k^\circ},
  \end{equation}
  where we defined \(v_k^\circ := (v_k, 0)\) and \(u_k^\circ := (0, u_k)\), and we abusively write
  \(
    \arraycolsep=1.4pt
    \bmat{w_1 & \cdots & w_k}
  \)
  instead of
  \(
    \arraycolsep=1.4pt
    \bmat{v_1^\circ & u_1^\circ & \cdots & v_k^\circ & u_k^\circ}
  \).
  The projection of $K_0$ into the block-Krylov subspace $\Span\{w_1,\ldots, w_k\} := \Span\{v_1^\circ, u_1^\circ, \ldots, v_k^\circ, u_k^\circ\}$ is also shuffled to block-Hessenberg form with blocks of size \(2\).
  Indeed, if we multiply~\eqref{eq:mo-ba} on the right with \(P_k\) and use~\eqref{eq:form-Wk}, we obtain
  \begin{equation}
  \label{eq:ba-A-B}
    K_0 W_k =
    \bmat{V_{k+1} & 0 \\ 0 & U_{k+1}} P_{k+1} P_{k+1}^T
    \bmat{0 & H_{k+1,k} \\ F_{k+1,k} & 0} P_k
     =
    W_{k+1} G_{k+1,k},
  \end{equation}
  where
  $$G_{k+1,k}
  =
      \begin{bmatrix}
        \Psi_{1,1} & \Psi_{1,2} & \hdots & \Psi_{1,k}    \\
        \Psi_{2,1} & \Psi_{2,2} & \ddots & \vdots       \\
                   & \ddots     & \ddots & \Psi_{k-1,k} \\
                   &            & \ddots & \Psi_{k,k}   \\
                   &            &        & \Psi_{k+1,k}
      \end{bmatrix},
  \qquad
  \Psi_{i,j} = \begin{bmatrix} 0 & h_{i,j} \\ f_{i,j} & 0 \end{bmatrix}.$$
  The two relations at line~\ref{alg:mo:init} of \Cref{alg:mo} can be rearranged as
  \begin{equation}
    \label{eq:ba-bc}
    \bmat{v_1 & 0 \\ 0 & u_1} \bmat{\beta & 0 \\ 0 & \gamma} = \bmat{b & 0 \\ 0 & c}\quad \Longleftrightarrow \quad w_1 \Gamma= D.
  \end{equation}
  Identities~\eqref{eq:ba-A-B} and~\eqref{eq:ba-bc} characterize the block-Arnoldi process applied to \(K_0\) with initial block \(D\).
  We summarize the process as \Cref{alg:block_arnoldi} where all \(w_k \in \R^{(n + m) \times 2}\) and \(\Psi_{i,k} \in \R^{2 \times 2}\) are determined such that both $w_k\T w_k = I_2$ and the equations on lines~\ref{alg:block_arnoldi:init},~\ref{alg:block_arnoldi:psi} and~\ref{alg:block_arnoldi:recurrence} are verified.

  \begin{algorithm}[ht]
    \caption{%
      Block-Arnoldi Process
    }
    \label{alg:block_arnoldi}
    \begin{algorithmic}[1]
      \Require $K_0$, $D$
      \State\label{alg:block_arnoldi:init}%
      $w_1 \Gamma = D$
      \For{$k$ = 1,~2,~\(\ldots\)}
        \For{$i$ = 1,~\(\ldots\),~$k$}
          \State\label{alg:block_arnoldi:psi}%
          $\Psi_{i,k} = w_i\T K_0 w_k$
        \EndFor
        \State\label{alg:block_arnoldi:recurrence}%
        $w_{k+1} \Psi_{k+1,k} = K_0 w_k - \sum_{i=1}^k w_i \Psi_{i,k}$
      \EndFor
    \end{algorithmic}
  \end{algorithm}

  \subsection{Regularization of the block-Arnoldi process}

  Merging~\eqref{eq:mo-V}--\eqref{eq:mo-U} gives
  \begin{equation}
    \label{eq:mo-ba}
    \bmat{0 & A \\ B & 0}
    \bmat{V_k & 0 \\ 0 & U_k}
    =
    \bmat{V_{k+1} & 0 \\ 0 & U_{k+1}}
    \bmat{0 & H_{k+1,k} \\ F_{k+1,k} & 0},
  \end{equation}
  which is reminiscent of the relation one would obtain from applying an orthogonalization process to \(K_0\).
  Because \(K = K_0 + \blkdiag(\lambda I, \mu I)\),~\eqref{eq:mo-ba} yields
  \begin{align}
    \bmat{\lambda I & A \\ B & \mu I}
    \bmat{V_k & 0 \\ 0 & U_k}
    & =
    \left(
      \bmat{0 & A \\ B & 0}
      +
      \bmat{\lambda I & 0 \\ 0 & \mu I}
    \right)
    \bmat{V_k & 0 \\ 0 & U_k}
    \nonumber
    \\ & =
    \bmat{V_k & 0 \\ 0 & U_k}
    \bmat{\lambda I & H_k \\ F_k & \mu I}
    +
    \bmat{v_{k+1} & 0 \\ 0 & u_{k+1}}
    \bmat{0 & h_{k+1,k} e_k^T \\ f_{k+1,k} e_k^T & 0}
    \label{eq:mo-regularized}
  \end{align}

  The same reasoning applied to~\eqref{eq:ba-A-B} yields the following result, which parallels \citet[Theorem~\(2.1\)]{montoison-orban-2021}.

  \begin{theorem}
    \label{theorem:sparcity-block-arnoldi}
    Given the matrix \(K\) defined in~\eqref{eq:K} and the block right-hand side \(D\) defined in~\eqref{eq:def-K-K0-D}, the Krylov basis $W_k = \bmat{w_1 & \cdots & w_k}$ generated by \Cref{alg:block_arnoldi} with regularization has the form~\eqref{eq:form-Wk} where the vectors \(u_k\) and \(v_k\) are the same as those generated by \Cref{alg:mo} with initial vectors \(b\) and \(c\).
    In addition,
    \begin{equation}
      \label{eq:gsp-block-arnoldi}
      K
      W_k
      =
      W_{k+1} S_{k+1,k},
      \qquad
      S_{k+1,k} :=
      \begin{bmatrix}
        \Theta_{1,1} & \Psi_{1,2}   & \hdots & \Psi_{1,k}    \\
        \Psi_{2,1}   & \Theta_{2,2} & \ddots & \vdots       \\
                     & \ddots       & \ddots & \Psi_{k-1,k} \\
                     &              & \ddots & \Theta_{k,k} \\
                     &              &        & \Psi_{k+1,k}
      \end{bmatrix},
    \end{equation}
    where
    \[
      \Theta_{j,j} =
      \begin{bmatrix}
        \lambda & h_{j,j} \\
        f_{j,j} & \mu
      \end{bmatrix}
      \quad\! \text{and} \!\quad
      \Psi_{i,j} =
      \begin{bmatrix}
        0 & h_{i,j} \\
        f_{i,j} & 0
      \end{bmatrix}, \!\quad\!
      j = 1, \ldots, k, \!\quad\!
      i = 1, \ldots, j+1, \!\quad\!
      i \ne j.
    \]
    The scalars $h_{i,j}$, $f_{i,j}$ are those generated by \Cref{alg:mo} applied to $A$ and $B$ with initial vectors $b$ and $c$.
  \end{theorem}

  \begin{proof}
    \Cref{alg:block_arnoldi} applied to $K_0$ generates sparse pairs $w_k$ as in~\eqref{eq:form-Wk} because of the equivalence with \Cref{alg:mo}.
    The term \(\blkdiag(\lambda I, \mu I)\) can be seen as a regularization term:
    \begin{equation}
      \label{eq:regularization}
      \begin{bmatrix}
        \lambda I & 0 \\
        0 & \mu I
      \end{bmatrix}
      w_k
      =
      w_k \Lambda
      \quad \text{with} \quad
      \Lambda :=
      \begin{bmatrix}
        \lambda & 0 \\
        0 & \mu
      \end{bmatrix}.
    \end{equation}
    The identities~\eqref{eq:ba-A-B} and~\eqref{eq:regularization} allow us to write
    \begin{equation}
      \label{eq:regularization_gsp}
      K
      W_k
      =
      W_{k+1}
      \begin{bmatrix}
        \Psi_{1,1} + \Lambda & \Psi_{1,2} & \hdots       & \Psi_{1,k}           \\
        \Psi_{2,1}           & \ddots     & \ddots       & \vdots               \\
                             & \ddots     & \ddots       & \Psi_{k-1,k}         \\
                             &            & \Psi_{k,k-1} & \Psi_{k,k} + \Lambda \\
                             &            &              & \Psi_{k+1,k}
      \end{bmatrix},
    \end{equation}
    which amounts to~\eqref{eq:gsp-block-arnoldi} because $\Theta_{k,k} = \Psi_{k,k} + \Lambda$.
  \end{proof}

  Note that~\eqref{eq:gsp-block-arnoldi} is identical to~\eqref{eq:mo-regularized} where the order of the \(w_k\) has been permuted according to \(P_k\).

  Because of \cref{theorem:sparcity-block-arnoldi}, the Krylov basis $W_k$ generated by \Cref{alg:block_arnoldi} must have the sparsity structure~\eqref{eq:form-Wk}, so that only \(u_k\) and \(v_k\) need be generated, and they may be generated directly from \Cref{alg:mo}.
  The key point is that generating orthonormal bases of $\mathcal{K}_k(K, d)$ and $\mathcal{K}_k(K, D)$ by the Arnoldi process and \Cref{alg:mo}, respectively, require exactly the same amount of storage and $\mathcal{K}_k(K, d) \subset \mathcal{K}_k(K, D)$.
  Thus, residual norms produced by \GMRES are certain to be at least as large as those generated by a minimum-residual method that seeks an approximate solution $x_k$ in $\mathcal{K}_k(K, D)$.
  Such a method is the subject of the next section.

  \section{Derivation of \GPMR}

  In this section, we develop the method \GPMR based upon \Cref{alg:mo} with regularization to solve~\eqref{eq:gsp_preconditioned} in which the \(k\)-th iterate has the form
  \begin{equation}
    \label{eq:xy}
    \begin{bmatrix} x_k \\ y_k \end{bmatrix} = W_k z_k,
  \end{equation}
  where $z_k \in \R^{2k}$.
  Thanks to~\eqref{eq:ba-bc} and~\eqref{eq:gsp-block-arnoldi}, the residual can be written
  \begin{align}
    r_k & =
    \begin{bmatrix}
      b \\ c
    \end{bmatrix}
    -
    \begin{bmatrix}
      \lambda I & A \\
      B & \mu I
    \end{bmatrix}
    \begin{bmatrix}
      x_k \\ y_k
    \end{bmatrix}
    \nonumber \\
    & =
     w_1 \bmat{\beta \\ \gamma} - W_{k+1} S_{k+1,k} z_k
    \nonumber \\
    & = W_{k+1} ( \beta e_1 + \gamma e_2 - S_{k+1,k} z_k ).
    \label{eq:residual_general}
  \end{align}
  Because $W_{k+1}$ has orthonormal columns, $\|r_k\|$ can be minimized by defining $z_k$ as the solution of the linear least-squares problem
  \begin{equation}
    \label{eq:sub-gpmr}
    \minimize{z_k \in \R^{2k}} \|S_{k+1,k} z_k - (\beta e_1 + \gamma e_2)\|.
  \end{equation}

  \subsection{Relation between \GPMR and \BLGMRES}

  The \(k\)-th \BLGMRES iterate is defined by the matrix linear least-squares problem
  \begin{equation}
    \label{eq:problem-blgmres}
    \minimize{} \left\|
    \bmat{b & 0 \\ 0 & c}
    -
    \bmat{\lambda I & A \\ B & \mu I}
    \bmat{x_k^b & x_k^c \\ y_k^b & y_k^c}
    \right\|_F
  \end{equation}
  where $(x_k^b,~y_k^b) = W_k z_k^b$ and $(x_k^c,~y_k^c) = W_k z_k^c$.
  Accordingly, the \(k\)-th \BLGMRES subproblem is
  \begin{equation}
  \label{eq:sub-blgmres}
  \minimize{z_k^b, z_k^c \in \R^{2k}} \left\|S_{k+1,k} \bmat{z_k^b & z_k^c} - \bmat{\beta e_1 & \gamma e_2}\right\|_F,
  \end{equation}
  so that $z_k^b$ and $z_k^c$ solve the subproblem associated with right-hand sides $\beta e_1$ and $\gamma e_2$.
  In exact arithmetic, the solutions of~\eqref{eq:sub-gpmr} and~\eqref{eq:sub-blgmres} are connected via $z_k = z_k^b + z_k^c$, and the \GPMR and \BLGMRES approximations are connected via $x_k = x_k^b + x_k^c$ and $y_k = y_k^b + y_k^c$.
  We now outline the main stages for solving~\eqref{eq:sub-gpmr}.

  \subsection{A QR factorization}

  The solution of~\eqref{eq:sub-gpmr} can be determined via the QR factorization
  \begin{equation}
    \label{eq:qr_Sk}
    S_{k+1, k} = Q_k \begin{bmatrix} R_k \\ 0 \end{bmatrix},
  \end{equation}
  which can be updated at each iteration, where $Q_k \in \R^{(2k+2) \times (2k+2)}$ is a product of Givens reflections, and $R_k \in \R^{(2k) \times (2k)}$ is upper triangular.
  At each iteration, four new reflections are necessary to update~\eqref{eq:qr_Sk}.
  We denote their product $Q_{2k-1, 2k+2}$ so that $Q_k\T = Q_{2k-1,2k+2} \ldots Q_{1,4}$.
  For $i = 1,\ldots,k$, the structure of $Q_{2i-1, 2i+2}$ is
  $$
  \kbordermatrix{
         & 1 & \hdots & 2i-2 & 2i-1 & 2i & 2i+1 & 2i+2 & 2i+3 & \hdots & 2k+2 \\
  1      & 1 &        &      &      &    &      &      &      &        &      \\
  \vdots &   & \ddots &      &      &    &      &      &      &        &      \\
  2i-2   &   &        & 1    &      &    &      &      &      &        &      \\
  2i-1   &   &        &      & \x   & \x & \x   & \x   &      &        &      \\
  2i     &   &        &      & \x   & \x & \x   & \x   &      &        &      \\
  2i+1   &   &        &      & \x   & \x & \x   & \x   &      &        &      \\
  2i+2   &   &        &      & \x   & \x & \x   & \x   &      &        &      \\
  2i+3   &   &        &      &      &    &      &      & 1    &        &      \\
  \vdots &   &        &      &      &    &      &      &      & \ddots &      \\
  2k+2   &   &        &      &      &    &      &      &      &        & 1    \\
  }
  $$
  where the diagonal block extracted from rows and columns $2i-1,\ldots,2i+2$ is the product of the following four Givens reflections
  $$
    \bmat{
      1 &         &                     &   \\
        & c_{4,i} & \phantom{-} s_{4,i} &   \\
        & s_{4,i} &          -  c_{4,i} &   \\
        &         &                     & 1
    }
    \!\!
    \bmat{
      1 &         &      &                     \\
        & c_{3,i} &      & \phantom{-} s_{3,i} \\
        &         & 1    &                     \\
        & s_{3,i} &      &          -  c_{3,i}
    }
    \!\!
    \bmat{
      c_{2,i} & \phantom{-} s_{2,i} &      &   \\
      s_{2,i} &          -  c_{2,i} &      &   \\
              &                     & 1    &   \\
              &                     &      & 1
    }
    \!\!
    \bmat{
      c_{1,i} &    &      & \phantom{-} s_{1,i} \\
              & 1  &      &                     \\
              &    & 1    &                     \\
      s_{1,i} &    &      &          -  c_{1,i}
  }.
  $$
  The result $(\aout_1, \aout_2, \aout_3, \aout_4)$ of a matrix-vector product between the above \(4\)\(\times\)\(4\) block and a vector $(\ainshort_1, \ainshort_2, \ainshort_3, \ainshort_4)$ can be obtained via \Cref{alg:ref}.

  \begin{algorithm}[ht]
    \caption{%
      Procedure ref
    }
    \label{alg:ref}
    \begin{algorithmic}[1]
      \Require $i$, $\ainshort_1$, $\ainshort_2$, $\ainshort_3$, $\ainshort_4$
      \State $t = c_{1,i} \ain_1 + s_{1,i} \ain_4$, $\aout_4 = s_{1,i} \ain_1 - c_{1,i} \ain_4$, $\aout_1 = t$ \Comment{first reflection}
      \State $t = c_{2,i} \aout_1 + s_{2,i} \ain_2$, $\aout_2 = s_{2,i} \aout_1 - c_{2,i} \ain_2$, $\aout_1 = t$ \Comment{second reflection}
      \State $t = c_{3,i} \aout_2 + s_{3,i} \aout_4$, $\aout_4 = s_{3,i} \aout_2 - c_{3,i} \aout_4$, $\aout_2 = t$ \Comment{third reflection}
      \State $t = c_{4,i} \aout_2 + s_{4,i} \ain_3$, $\aout_3 = s_{4,i} \aout_2 - c_{4,i} \ain_3$, $\aout_2 = t$ \Comment{fourth reflection}
    \end{algorithmic}
  \end{algorithm}

  At iteration $k$, \cref{alg:mo} generates two new columns, and to update the QR decomposition we need first to apply all previous reflections as follows
  \[
  Q_{k-1}\T
  \bmat{
    \Psi_{1,k}   \\
    \vdots       \\
    \Psi_{k-1,k} \\
    \Theta_{k,k} \\
    \Psi_{k+1,k}
  }
  =
  Q_{2k-5,2k-2} \ldots Q_{3,6}
  \left[
    \begin{array}{cc}
      r_{1,2k-1}       & r_{1,2k}       \\
      r_{2,2k-1}       & r_{2,2k}       \\
      \bar{r}_{3,2k-1} & \bar{r}_{3,2k} \\
      \bar{r}_{4,2k-1} & \bar{r}_{4,2k} \\
      \multicolumn{2}{c}{\Psi_{3,k}}    \\
      \multicolumn{2}{c}{\vdots}        \\
      \multicolumn{2}{c}{\Psi_{k+1,k}}  \\
    \end{array}
  \right]
  =
  \bmat{
    r_{1,2k-1}          & r_{1,2k}          \\
    \vdots              & \vdots            \\
    r_{2k-2,2k-1}       & r_{2k-2,2k}       \\
    \bar{r}_{2k-1,2k-1} & \bar{r}_{2k-1,2k} \\
    \bar{r}_{2k,2k-1}   & \bar{r}_{2k,2k}   \\
                        & h_{k+1,k}         \\
    f_{k+1,k}           &
  },
  \]
  and then compute and apply the four reflections that constitute $Q_{2k-1,2k+2}$ such that coefficients under the diagonal are zeroed out
  $$
  Q_{2k-1,2k+2}
  \bmat{
    r_{1,2k-1}          & r_{1,2k}          \\
    \vdots              & \vdots            \\
    r_{2k-2,2k-1}       & r_{2k-2,2k}       \\
    \bar{r}_{2k-1,2k-1} & \bar{r}_{2k-1,2k} \\
    \bar{r}_{2k,2k-1}   & \bar{r}_{2k,2k}   \\
                        & h_{k+1,k}         \\
    f_{k+1,k}           &
  }
  =
  \bmat{
    r_{1,2k-1}    & r_{1,2k}    \\
    \vdots        & \vdots      \\
    r_{2k-2,2k-1} & r_{2k-2,2k} \\
    r_{2k-1,2k-1} & r_{2k-1,2k} \\
    0             & r_{2k,2k}   \\
    0             & 0           \\
    0             & 0
  }.
  $$
  A procedure to compute the Givens sines and cosines, and finalize the QR factorization of $S_{k+1,k}$ is described as \Cref{alg:givens}.
  \begin{algorithm}[ht]
    \caption{%
      Procedure givens
    }
    \label{alg:givens}
    \begin{algorithmic}[1]
      \Require $k$, $\bar{r}_{2k-1,2k-1}$, $\bar{r}_{2k-1,2k}$, $\bar{r}_{2k,2k-1}$, $\bar{r}_{2k,2k}$, $h_{k+1,k}$, $f_{k+1,k}$
      \State $\barbar{r}_{2k-1,2k-1} = (\bar{r}_{2k-1,2k-1}^2 + f_{k+1,k}^2)^{\frac12}$ \Comment{annihilate $f_{k+1,k}$}
      \State $c_{1,k} = \bar{r}_{2k-1,2k-1} / \barbar{r}_{2k-1,2k-1}$,
             $s_{1,k} =           f_{k+1,k} / \barbar{r}_{2k-1,2k-1}$
      \State $\barbar{r}_{2k-1,2k} = c_{1,k} \bar{r}_{2k-1,2k}$
      \State $   \bar{r}_{2k+2,2k} = s_{1,k} \bar{r}_{2k-1,2k}$
      \State $r_{2k-1,2k-1} = (\barbar{r}_{2k-1,2k-1}^2 + \bar{r}_{2k,2k-1}^2)^{\frac12}$ \Comment{annihilate $\bar{r}_{2k,2k-1}$}
      \State $c_{2,k} = \barbar{r}_{2k-1,2k-1} / r_{2k-1,2k-1}$,
             $s_{2,k} =    \bar{r}_{2k  ,2k-1} / r_{2k-1,2k-1}$
      \State $        {r}_{2k-1,2k} = c_{2,k} \barbar{r}_{2k-1,2k} + s_{2,k} \bar{r}_{2k,2k}$
      \State $ \barbar{r}_{2k  ,2k} = s_{2,k} \barbar{r}_{2k-1,2k} - c_{2,k} \bar{r}_{2k,2k}$
      \State $\mathring{r}_{2k,2k} = (\barbar{r}_{2k,2k}^2 + \bar{r}_{2k+2, 2k}^2)^{\frac12}$ \Comment{annihilate $\bar{r}_{2k+2, 2k}$}
      \State $c_{3,k} = \barbar{r}_{2k  ,2k} / \mathring{r}_{2k,2k}$,
             $s_{3,k} =    \bar{r}_{2k+2,2k} / \mathring{r}_{2k,2k}$
      \State $r_{2k,2k} = (\mathring{r}_{2k,2k}^2 + h_{k+1,k}^2)^{\frac12}$ \Comment{annihilate $h_{k+1,k}$}
      \State $c_{4,k} = \mathring{r}_{2k,2k} / r_{2k,2k}$,
             $s_{4,k} =            h_{k+1,k} / r_{2k,2k}$
    \end{algorithmic}
  \end{algorithm}
  Note that the first parameter of \Cref{alg:ref} and \Cref{alg:givens} is used to define which Givens sines and cosines are read from or written to memory.

  \subsection{\GPMR iterate and residual norm computation}

  We have from~\eqref{eq:residual_general} and~\eqref{eq:qr_Sk}:
  \begin{equation}
    \label{eq:residual_gblmr}
    \|r_k\| =
    \left \|Q_k  \bmat{R_k \\ 0} z_k - (\beta e_1 + \gamma e_2)\right\| =
    \left\| \bmat{R_k \\ 0} z_k - \bar{t}_k \right\|,
  \end{equation}
  where \(\bar{t}_k := Q_k^T (\beta e_1 + \gamma e_2) = (t_k, \bar{\tau}_{2k+1}, \bar{\tau}_{2k+2})\), $t_k := (\tau_1, \ldots, \tau_{2k})$ represents the first \(2k\) components of \(\bar{t}_k\), and the recurrence starts with \(\bar{t}_0 := (\bar{\tau}_1, \bar{\tau}_2) = (\beta, \gamma)\).
  $\bar{t}_k$ can be easily determined from $\bar{t}_{k-1}$ because $\bar{t}_k = Q_{2k-1,2k+2} (\bar{t}_{k-1}, 0, 0)$.
  The solution of~\eqref{eq:sub-gpmr} is thus $z_k := (\zeta_1, \ldots, \zeta_{2k})$ found by solving $R_k z_k = t_k$ with backward substitution.

  The definitions of \(\bar{t}_k\) and \(z_k\) together with~\eqref{eq:residual_gblmr} yield
  \begin{equation}
    \label{eq:residual_norm_gblmr}
    \|r_k\| =
    \sqrt{\bar{\tau}_{2k+1}^2 + \bar{\tau}_{2k+2}^2}.
  \end{equation}

  As in \GMRES, we only compute $z_k$ when \(\|r_k\|\) is smaller than a user-provided threshold.
  Thanks to~\eqref{eq:xy}, the solution may be computed efficiently as
  \begin{subequations}
    \label{eq:sol_gblmr}
    \begin{align}
    x_k &= \textstyle\sum_{i=1}^k \zeta_{2i-1} v_i, \\
    y_k &= \textstyle\sum_{i=1}^k \zeta_{2i} u_i.
    \end{align}
  \end{subequations}
  We summarize the complete procedure as \Cref{alg:gpmr}.

  \begin{algorithm}[t]
    \caption{%
      \GPMR
    }
    \label{alg:gpmr}
    \begin{algorithmic}[1]
      \Require $A$, $B$, $b$, $c$, $\lambda$, $\mu$, $\epsilon > 0$, $k_{\max} > 0$
      \State $\beta v_1 = b$, $\gamma u_1 = c$ \Comment{$(\beta, \gamma) > 0$ so that $\|v_1\| = \|u_1\| = 1$}
      \State $\bar{\tau}_1 = \beta$, $\bar{\tau}_2 = \gamma$ \Comment{Initialize $\bar{t}_0$}
      \State $\|r_0\| = (\bar{\tau}_1^2 + \bar{\tau}_2^2)^{\frac12}$ \Comment{compute $\|r_0\|$}
      \State $k = 0$
      \While{\(\|r_k\| > \epsilon\) \textbf{and} \(k < k_{\max}\)}
        \State $k \leftarrow k + 1$
        \State $q = A u_k$ \Comment{Orthogonal Hessenberg reduction}
        \State $p = B v_k$
        \For{$i$ = 1,~\(\ldots\),~$k$}
          \State $h_{i,k} = v_i\T q$
          \State $f_{i,k} = u_i\T p$
          \State $q = q - h_{i,k} v_i$
          \State $p = p - f_{i,k} u_i$
        \EndFor
        \State $h_{k+1,k} v_{k+1} = q$ \Comment{$h_{k+1,k} > 0$ so that $\|v_{k+1}\| = 1$}
        \State $f_{k+1,k} u_{k+1} = p$ \Comment{$f_{k+1,k} > 0$ so that $\|u_{k+1}\| = 1$}
        \State $\bar{r}_{1,2k} = h_{1,k}$, $\bar{r}_{2,2k-1} = f_{1,k}$
        \State \textbf{if} $k \ne 1$ \textbf{then} $(\bar{r}_{1,2k-1}, \bar{r}_{2,2k}) = (0, 0)$ \textbf{else} $(\bar{r}_{1,2k-1}, \bar{r}_{2,2k}) = (\lambda, \mu)$
        \For{$i = 1, \ldots, k-1$} \Comment{Apply $Q_{2k-5,2k-2},\ldots,Q_{1,4}$}
          \State \textbf{if} $i \ne k-1$ \textbf{then} $(\rho, \delta) = (0, 0)$ \textbf{else} $(\rho, \delta) = (\lambda, \mu)$
            \State $r_{2i-1,2k-1}, r_{2i,2k-1}, \bar{r}_{2i+1,2k-1}, \bar{r}_{2i+2,2k-1} = \applyGivens(i, \bar{r}_{2i-1,2k-1}, \bar{r}_{2i,2k-1}, \rho, f_{i+1,k})$
            \State $r_{2i-1,2k}, r_{2i,2k}, \bar{r}_{2i+1,2k}, \bar{r}_{2i+2,2k} = \applyGivens(i, \bar{r}_{2i-1,2k}, \bar{r}_{2i,2k}, h_{i+1,k}, \delta)$
        \EndFor
        \State $r_{2k-1,2k-1}, r_{2k-1,2k}, r_{2k,2k} =$ \Comment{Compute and apply $Q_{2k-1,2k+2}$} \par
        \hspace{\algorithmicindent} $\computeGivens(k, \bar{r}_{2k-1,2k-1}, \bar{r}_{2k-1,2k}, \bar{r}_{2k,2k-1}, \bar{r}_{2k,2k}, h_{k+1,k}, f_{k+1,k})$
        \State $\tau_{2k-1}, \tau_{2k}, \bar{\tau}_{2k+1}, \bar{\tau}_{2k+2} = \applyGivens(k, \bar{\tau}_{2k-1}, \bar{\tau}_{2k}, 0, 0)$ \Comment{update $\bar{t}_k$}
        \State $\|r_k\| = (\bar{\tau}_{2k+1}^2 + \bar{\tau}_{2k+2}^2)^{\frac12}$ \Comment{compute $\|r_k\|$}
      \EndWhile
        \State $\zeta_{2k} = \tau_{2k} / r_{2k,2k}$ \Comment{compute $z_k$}
        \For{$i = 2k-1, \ldots, 1$}
          \State $\zeta_i = (\tau_i - \sum_{j=i+1}^{2k} r_{i,j} \zeta_{j}) / r_{i,i}$
        \EndFor
        \State $x_k = \sum_{i=1}^k \zeta_{2i-1} v_i$ \Comment{compute $x_k$}
        \State $y_k = \sum_{i=1}^k \zeta_{2i} u_i$ \Comment{compute $y_k$}
    \end{algorithmic}
  \end{algorithm}

  \subsection{Memory requirements}

  \Cref{tab:memory} summarizes the storage costs of $k$ iterations of \GPMR, \GMRES and \BLGMRES.

  \begin{table}[ht]
    \label{tab:memory}
    \caption{Memory requirements for $k$ iterations of \GPMR, \GMRES and \BLGMRES.}
    \footnotesize
    \begin{center}
      \setlength{\arrayrulewidth}{1pt}
      \begin{tabular}{rccccccc}
        \hline
        & $(x_k, y_k)$ & $(q, p)$ & $(V_k, U_k)$ & $t_k$ & $z_k$ & $Q_k$ & $R_k$      \\
        \hline
        \GPMR     & $\phantom{2(}m+n\phantom{)}$        & $\phantom{2(}m+n\phantom{)}$    & $\phantom{2}k(m+n)$     & $2k$  & $2k$  & $8k$  & $k(2k+1)\phantom{/2}$  \\
        \GMRES    & $\phantom{2(}m+n\phantom{)}$        & $\phantom{2(}m+n\phantom{)}$    & $\phantom{2}k(m+n)$     & $\phantom{2}k$   & $\phantom{2}k$   & $2k$  & $k(\phantom{2}k+1)/2$ \\
        \BLGMRES  & $2(m+n)$     & $2(m+n)$ & $2k(m+n)$    & $4k$  & $4k$  & $8k$  & $k(2k+1)\phantom{/2}$  \\
        \hline
      \end{tabular}
    \end{center}
  \end{table}

  Some \GPMR variables are paired in \Cref{tab:memory} to easily identify their \GMRES and \BLGMRES counterparts.
  Note that $t_k$ and $z_k$ can share the same storage because $R_k t_k = z_k$ can be solved in-place.

  \section{Implementation and numerical experiments}
  \label{sec:num-results}

  We implemented \Cref{alg:gpmr} in Julia \citep{bezanson-edelman-karpinski-shah-2017}, version \(1.6\), as part of our \texttt{Krylov.jl} collection of Krylov methods \citep{montoison-orban-krylov-2020}.
  Our implementation of \GPMR is applicable in any floating-point system supported by Julia, and runs on CPU and GPU.
  The GPU support can be particularly relevant for~\eqref{eq:ism} because, as a Krylov method, \GPMR only requires linear operators that model $A_{\mathcal{I} \Gamma}u$, $B_{\Gamma \mathcal{I}}v$, $u \mapsto M_{\mathcal{I}\mathcal{I}} \backslash u$ and $v \mapsto N_{\Gamma\Gamma} \backslash v$.
  For instance, $v \mapsto N_{\Gamma\Gamma} \backslash v$ can be the forward and backward substitutions with the factors of an LU decomposition of $N_{\Gamma\Gamma}$.
  The use of abstract linear operators allows us to store $A_{\mathcal{I} \Gamma}$ and $B_{\Gamma \mathcal{I}}$ as well as decompositions of the diagonal blocks of~\eqref{eq:ism} on distinct compute nodes and leverage parallel architectures, such as GPUs.
  When the matrices are unstructured, \citet{duff-scott-2005} propose a robust arrowhead reordering such that each diagonal block is nonsingular and recovers a system of the form~\eqref{eq:ism}.

  We evaluate the performance of \GPMR on systems generated from unsymmetric matrices in the SuiteSparse Matrix Collection \citep{davis-hu-2011}.
  We use METIS to form a \(2\)\(\times\)\(2\) block matrix and use the two diagonal blocks to build a right block-Jacobi preconditioner $P_r$ with $\lambda = \mu = 1$.
  We set $P_{\ell} = I$ so the residual norm of~\eqref{eq:gsp} is identical to that of~\eqref{eq:gsp_preconditioned}.
  The right-hand side $(b_{\star},c_{\star})$ is generated so the exact solution of~\eqref{eq:gsp} is the vector of ones.
  We compare \GPMR to our implementation of \GMRES without restart in terms of number of iterations.
  Each algorithm stops as soon as $\|r_k\| \leq \varepsilon_a + \|(b,c)\| \varepsilon_r$ with absolute tolerance $\varepsilon_a = 10^{-12}$ and relative tolerance $\varepsilon_r = 10^{-10}$.
  \Cref{tab:ufl} summarizes our results, which show an improvement in terms of number of iterations ranging from about 10\% up to 50\% in favor of \GPMR.
  \Cref{fig:ufl1} reports residual histories of \GPMR, \GMRES and \BLGMRES where the two approximate solutions are summed on problems \emph{scircuit}, \emph{sme3Dc}, \emph{PR02R} and \emph{sherman5}.

  \begin{table}[ht]
    \label{tab:ufl}
    \caption{Number of iterations of \GPMR and \GMRES on systems from the SuiteSparse Matrix Collection.}
    \footnotesize
    \begin{center}
      \setlength{\arrayrulewidth}{1pt}
      \begin{tabular}{rrrrrr}
        \hline
        name            & size    & nnz      & \GMRES & \GPMR & gain \\
        \hline
        sherman5        & 3312    & 20793    & 25     & 20    & 20\% \\
        powersim        & 15838   & 67562    & 141    & 101   & 28\% \\
        Ill\_Stokes     & 20896   & 191368   & 59     & 54    & 9\%  \\
        sme3Dc          & 42930   & 3148656  & 127    & 78    & 39\% \\
        rma10           & 46835   & 2374001  & 48     & 41    & 15\% \\
        ecl32           & 51993   & 380415   & 58     & 42    & 28\% \\
        venkat50        & 62424   & 1717792  & 48     & 35    & 27\% \\
        poisson3Db      & 85623   & 2374949  & 56     & 50    & 11\% \\
        ifiss\_mat      & 96307   & 3599932  & 42     & 33    & 21\% \\
        hcircuit        & 105676  & 513072   & 47     & 37    & 21\% \\
        PR02R           & 161070  & 8185136  & 97     & 68    & 30\% \\
        scircuit        & 170998  & 958936   & 48     & 24    & 50\% \\
        transient       & 178866  & 961790   & 567    & 470   & 17\% \\
        ohne2           & 181343  & 11063545 & 50     & 39    & 22\% \\
        thermomech\_dK  & 204316  & 2846228  & 128    & 84    & 34\% \\
        marine1         & 400320  & 6226538  & 84     & 60    & 29\% \\
        Freescale1      & 3428755 & 18920347 & 456    & 344   & 25\% \\
        \hline
      \end{tabular}
    \end{center}
  \end{table}

  \begin{figure}[ht]
    \centering
    \includetikzgraphics[width=0.49\textwidth]{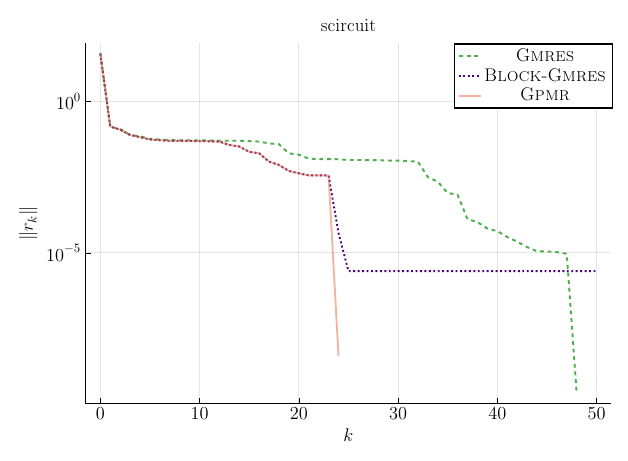}
    \hfill
    \includetikzgraphics[width=0.49\textwidth]{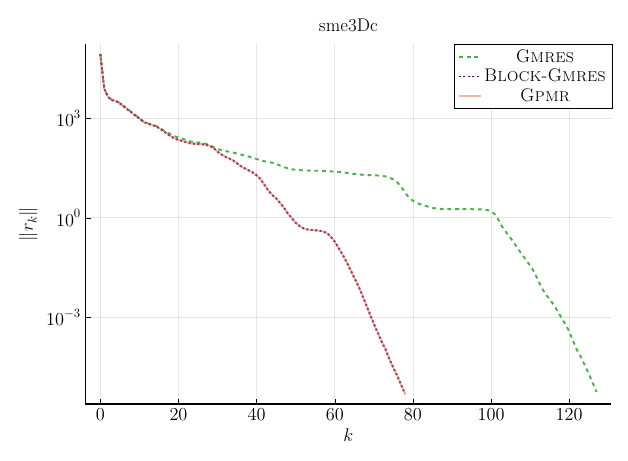}
    \\
    \includetikzgraphics[width=0.49\textwidth]{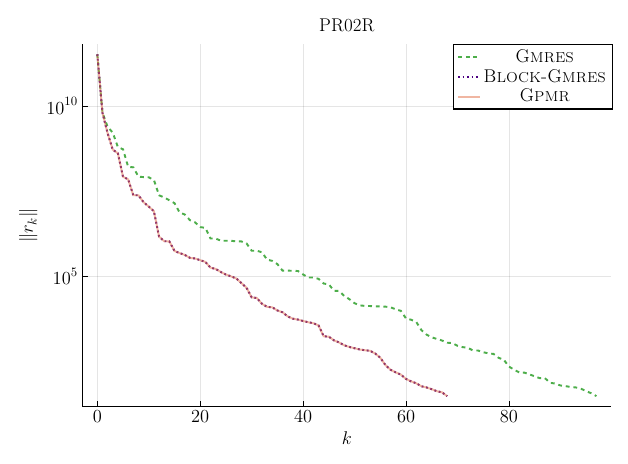}
    \hfill
    \includetikzgraphics[width=0.49\textwidth]{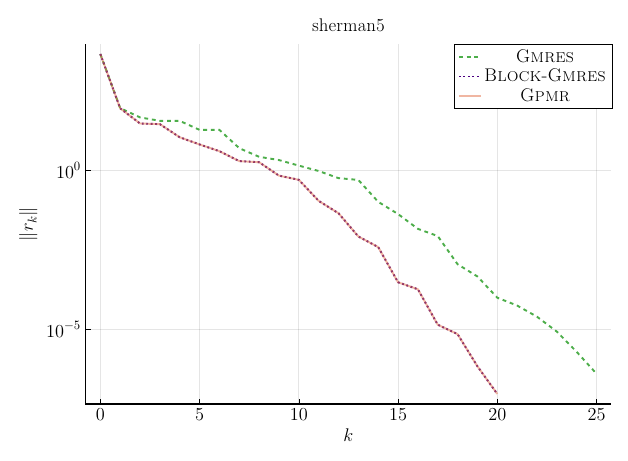}
    \label{fig:ufl1}
    \vspace{-12pt}
    \caption{Residual history of \GPMR, \GMRES and \BLGMRES.}
  \end{figure}

  The \GPMR and \BLGMRES residuals are nearly superposed except for \emph{scircuit}, on which \BLGMRES stagnates.
  The same phenomenon occurs on a generalized saddle point build using matrices \emph{well1033} as $A$ and \emph{illc1033} as $B$, $M = I$, $N = 0$, $\lambda = 1$ and $\mu = 0$.
  \Cref{fig:1033} reports residual histories of \GPMR, \GMRES and \BLGMRES on the generalized saddle point system in double and quadruple precision.
  Although theoretically equivalent, \GPMR appears to be less sensitive to arithmetic errors due to loss of orthogonality than its counterpart implementation based on \BLGMRES.
  Indeed, the number of \GPMR and \GMRES iterations is the same in double and quadruple precision.

  \begin{figure}[ht]
    \centering
    \includetikzgraphics[width=0.49\textwidth]{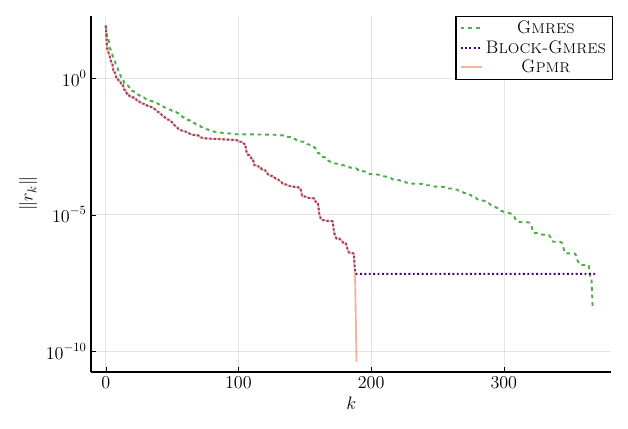}
    \hfill
    \includetikzgraphics[width=0.49\textwidth]{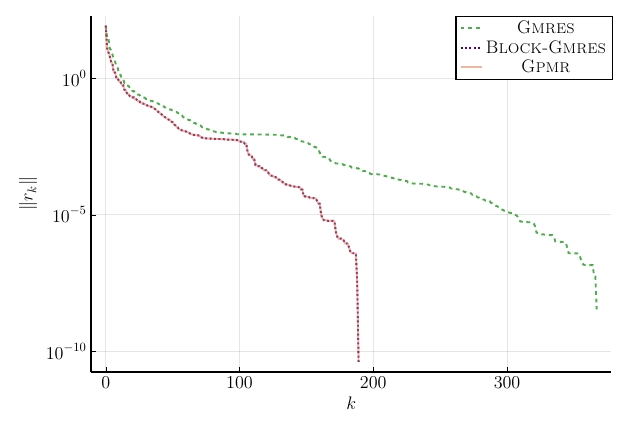}
    \label{fig:1033}
    \vspace{-12pt}
    \caption{Residual history of \GPMR, \GMRES and \BLGMRES on the generalized saddle point system in double (left) and quadruple precision (right).}
  \end{figure}

  When $K$, defined in~\eqref{eq:K}, is symmetric, \Cref{alg:mo} coincides with the orthogonal tridiagonalization process of \citet{saunders-simon-yip-1988} because $A\T=B$ and \GPMR is theoretically equivalent to \TriMR.
  We verify numerically the equivalence between the two methods on symmetric quasi-definite systems, with matrices $A$ from the SuiteSparse Matrix Collection, $M = N = I$, $\lambda = 1$ and $\mu = -1$.
  Each algorithm stops with the same tolerance as above.
  Because \GPMR can be viewed as \TriMR with full reorthogonalization, we use different floating-point systems to observe any loss of orthogonality in the Krylov basis.
  \Cref{fig:gpmr_trimr} reports residual histories of \GPMR in double precision and \TriMR in double, quadruple and octuple precision.
  The plots suggest that reorthogonalization is a more powerful device than extended precision.

  \begin{figure}[ht]
    \centering
    \includetikzgraphics[width=0.49\textwidth]{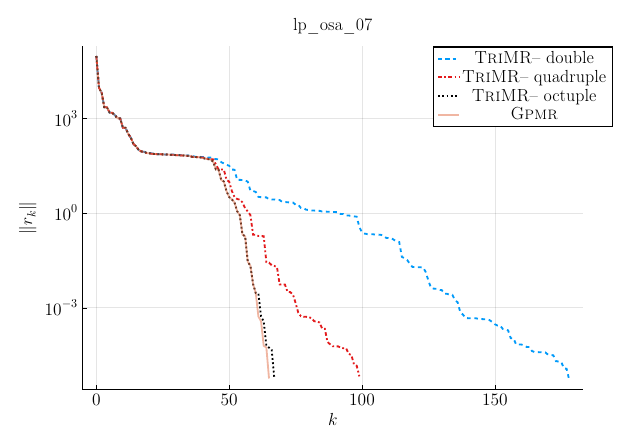}
    \hfill
    \includetikzgraphics[width=0.49\textwidth]{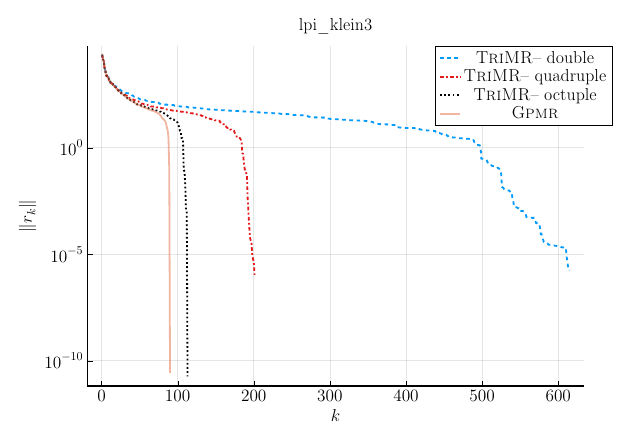}
    \label{fig:gpmr_trimr}
    \vspace{-12pt}
    \caption{Residual history of \GPMR and \TriMR.}
  \end{figure}

  \section{Discussion and extensions}

  Based upon \Cref{alg:mo}, it is possible to develop another method, \GPCG, in the spirit of \FOM \citep{saad-1981}.
  The \(k\)-th \GPCG iterate is defined by the Galerkin condition $W_k\T r_k = 0$.
  Its associated subproblem selects $z_k$ in~\eqref{eq:xy} as the solution of the square system
  \begin{equation}
    \label{eq:sub-gblcg}
    S_k z_k = \beta e_1 + \gamma e_2,
  \end{equation}
  where $S_k$ denotes the leading \((2k)\)\(\times\)\((2k)\) submatrix of $S_{k+1,k}$ in~\eqref{eq:gsp-block-arnoldi}.
  However, \GPCG may break down if $S_k$ is singular, and in that respect shares the disadvantages of \FOM, whereas the \GPMR iterates are always well defined.
  \GPCG could still be relevant for unsymmetric structured and positive-definite linear systems, such as those arising from the finite-element discretization of advection-diffusion equations \citep{xuemin-li-2009}, where $S_k$ is guaranteed to be nonsingular.
  Indeed, if $K$ is positive definite, its projection $S_k = W_k\T K W_k$ into the \(k\)-th Krylov subspace is also positive definite, which ensures that~\eqref{eq:sub-gblcg} has a unique solution.
  The same observation holds for \FOM and \BiCG \citep{fletcher-1976}, which should be restricted to certain classes of linear systems to avoid breakdowns.

  Although the focus of \GPMR is on unsymmetric linear systems, \Cref{fig:gpmr_trimr} shows that it is also relevant for ill-conditioned symmetric linear systems.
  Moreover, \GPMR allows to solve symmetric partitioned systems with symmetric indefinite blocks $M$ and $N$, whereas \TriMR requires them to be zero or definite matrices.

  A variant with restart in the spirit of \GMRESk is easily implemented on top of \GPMR.
  A limited-memory variant of \GPMR can be also developed and compared to \DQGMRES\citep{saad-wu-1996}.
  We leave the investigation of such extension to future work.

  \small
  \bibliographystyle{abbrvnat}
  \bibliography{abbrv,gspmr}
  \normalsize


\end{document}